\documentclass[11pt]{amsart}

\usepackage{times,amsfonts,amsmath,amstext,amsbsy,amssymb,
amsopn,amsthm,upref,eucal}
\usepackage[T1]{fontenc}
\usepackage{marginnote}
\usepackage{yfonts}
\usepackage{color}

\newcommand \R {{ \mathbb R}}

\newcommand \Z {{ \mathbb Z}}
\newcommand \N {{ \mathbb N}}

\newcommand{\vol} {{\operatorname{vol} }}

\newcommand \re {{ \operatorname{Re} }}

\newcommand{\SL}{\operatorname{SL}}
\newcommand{\PSL}{\operatorname{PSL}}
\newcommand{\ASL}{\operatorname{ASL}}
\renewcommand{\sl}{\operatorname{\mathfrak s\mathfrak l}}

\newtheorem{theorem}{Theorem}[section]
\newtheorem {lemma} {Lemma}[section]

\newtheorem{remark}{Remark}[section]

\title[Effective equidistribution of coordinate horocycle flows]
{Effective equidistribution for some unipotent flows in $\PSL(2, \R)^k$ mod cocompact, irreducible lattice}
\author{James Tanis}
\begin{document}

\begin{abstract}
Let $d \geq 2$, and let $\Gamma \subset \PSL(2, \R)^d$ be an irreducible, cocompact lattice.   
We prove a sharp estimate up to a logarithmic factor on the rate of equidistribution of 
coordinate horocycle flows on $\Gamma \backslash\PSL(2, \R)^d$.  

\end{abstract}
\maketitle
\section{Introduction}

There has been greater interest recently in making Ratner's equidistribution theorems effective (see \cite{R1} and  \cite{R2}).  
Green-Tao proved all Diophantine nilflows 
on any nilmanifold become equidistributed at polynomial speed, see \cite{GT}. 
Flaminio-Forni proved rather sharp estimates on the speed of 
equidistribution for a class of higher step nilmanifolds.   
Einsiedler-Margulis-Venkatesh proved effective equidistribution 
for large closed orbits of semisimple groups 
on homogeneous spaces, 
under some technical restrictions, in \cite{EMV}.  

Effective equidistribution, at an unspecified rate, of coordinate horocycle flows on $\Gamma \backslash \PSL(2, \R)^d$
can be derived from a recent and general quantitative mixing 
result of Bj$\ddot{\textrm{o}}$rklund, Einsiedler and Gorodnik, see \cite{BEG}.  
In what follows, we prove sharp estimates, up to a logarithmic factor, 
for the equidistribution of coordinate horocycle flows 
in compact $\Gamma \backslash \PSL(2, \R)^d$.  

These flows are one-dimensional and non-horospherical unipotent flows 
on a (non-solvable) homogeneous space.    
The first effective equidistribution result in this setting 
is due to Venkatesh on the product flow given by the 
horocycle flow and a circle translation on compact 
$\Gamma \backslash \SL(2, \R) \times \R / \Z$, see \cite{V}.   
Recently, the author and Vishe refined 
his approach and established a sharper 
rate of equidistribution that is independent of the spectral gap of the Laplacian, see \cite{TV}.   
At the same time, Flaminio, Forni and the author obtained 
a still sharper estimate (also independent of the spectral gap)
for the equidistribution of this flow 
via a completely different method, see \cite{FFT}.  
Effective estimates for the equidistribution of horocycle 
lifts to $\ASL(2, \Z) \backslash \ASL(2, \R)$ were proved 
in a series of papers by 
Str$\ddot{\text{o}}$mbergsson, Browning and Vinogradov, and Vinogradov, 
see \cite{St}, \cite{BV} and \cite{IV}.  

The outline of the paper is as follows.  
We prove Theorem~\ref{theo:equidistribution} first 
by using existing works and the theory of 
unitary representations and invariant distributions.  
The main point in the argument is that the invariant 
distributions for coordinate horocycle flows 
are defined in unitary Sobolev representations of $\PSL(2, \R)$ on $L^2(\Gamma \backslash \PSL(2, \R)^d)$, 
so they are already well-understood by the work of Flaminio and Forni 
on the equidistribution of horocycle flows 
(see Theorem 1.1 and Theorem 1.4 of \cite{FF1}).  
In fact, their work implies there is a basis of invariant distributions 
for coordinate horocycle flows 
which are generalized eigendistributions for 
a corresponding coordinate geodesic flow.  
Using a recent result of Kelmer-Sarnak on 
the strong spectral gap property of irreducible, cocompact lattices in 
$\PSL(2, \R)^d$, we estimate the 
contribution of these  invariant distributions 
to the ergodic integral of a given coordinate horocycle flow 
by iteratively applying the relevant 
coordinate geodesic map, which is a method of Flaminio and Forni in \cite{FF1}. 
The rest of the ergodic integral is estimated by solving a cohomological equation, as in \cite{FF1}.   \\ 

Now we discuss the setting for our result in detail.  
Let $d \in \N_{\geq 2}$ and $\Gamma \subset \PSL(2, \R)^d$ be a cocompact and irreducible lattice, 
and let $M = \Gamma \backslash \PSL(2, \R)^d$. 
Let $(X,U,V)$ be a basis of $\sl_2(\R)$ given by
 \[
 X= \begin{pmatrix} {1}&0\\0& {-1}
 \end{pmatrix}, \quad U=\begin{pmatrix} 0&1\\0& 0
 \end{pmatrix}, \quad V=\begin{pmatrix} 0&0\\1& 0
 \end{pmatrix}.
 \]   
 For each $1 \leq j \leq d$, define the coordinate geodesic and coordinate 
 horocycle vector fields $(X_j, U_j, V_j)$ in $\sl(2, \R)^d$ by 
\[
\begin{aligned}
&X_j := (0, \ldots, 0, X, 0, \ldots, 0) \,, \\
&U_j := (0, \ldots, 0, U, 0, \ldots, 0)\,, \\
&V_j := (0, \ldots, 0, V, 0, \ldots, 0)\,,  \\
\end{aligned}
\]
where the vector fields $X$, $U$ and $V$ respectively appear in the $j_{th}$ position 
of the tuple.  
 These matrices satisfy the commutation relations
 \begin{equation}\label{eq:comm_j}
 [X_j,U_j]= U_j , \quad [X_j,V_j]= -V_j , \quad [U_j,V_j]= 2X_j \,.
 \end{equation}
Now fix an integer $i \leq d$, and denote the coordinate horocycle flow, $\{h_t\}_{t\in \R}$, on $\Gamma \backslash \PSL(2, \R)^d$ by   
\[
h_t(x) = x e^{t U_i}\,,
\]
for any $x \in M$.  

Let $L^2(M)$ be the separable Hilbert space of 
complex-valued square-integrable functions on $M$ 
with respect to the Haar measure $\vol$.   
Let $C^\infty(M)$ be the space of smooth functions on $M$ 
and let $\mathcal D'(M) = \left(C^\infty(M)\right)'$ 
be its distributional dual space.  
Any element of the Lie algebra $\sl(2, \R)^d$ acts 
on $\mathcal D'(M)$ via the right regular representation.  

For each $j$, the center of the enveloping algebra of $\sl(2, \R)^d$ 
contains the second-order differential operator
$$
\Box_j := \left[- X_j^2 - 1/2(U_jV_j + V_jU_j)\right]\,.
$$
The Laplacian operator $\triangle$ is a second-order, 
elliptic element in the enveloping algebra of $\sl(2, \R)^d$.  
Moreover, it is an essentially self-adjoint differential 
operator on $L^2(M)$ that can be written 
$$
\triangle := \triangle_i + \triangle_{0}\,,
$$
where
$$
\triangle_i := -X_i^2 - 1/2(U_i^2 + V_i^2) \text{ and } \triangle_0 := - \sum_{j \neq i} X_j^2 + 1/2(U_j^2 + V_j^2)\,.
$$

Denote the inner product for $L^2(M)$ by $\langle \,, \rangle$\,.     
The Sobolev space of order $s \in \R^+$ is the maximal domain $W^s(M)$ of the inner product 
\[
\langle f, g\rangle_s := \langle (I  + \triangle)^s f, g\rangle\,,
\] 
where $I$ is the identity operator on $L^2(M)$.  Define $\Vert f \Vert_s^2 := \langle f, f\rangle_s$, and set $\Vert f \Vert := \Vert f \Vert_{0}$.  

The space of $s$-order distributions on $W^s(M)$ is $W^{-s}(M) = \left(W^s(M)\right)'$.  
Because $M$ is compact, 
$C^\infty(M) = \bigcap_{s > 0} W^s(M)$ and 
$\mathcal D'(M) = \cup_{s> 0} W^{-s}(M).$

We denote the space of distributions in $W^{-s}(M)$ 
that are invariant under $U_i$ by 
$$
\mathcal{I}^s(M) := \left\{D \in W^{-s}(M) :  U_i D = 0\right\}\,.
$$
Let $\mathcal I(M) := \bigcup_{s > 0} \mathcal I^s(M)$.  
By (\ref{equa:regular_to_pi}), 
the classification of $U_i$-invariant distributions in $\mathcal I(M)$ into irreducible, 
unitary representation spaces is given by the corresponding 
classification of horocycle flow-invariant distributions from Theorem 1.1 of \cite{FF1}.  

Let $\sigma_{pp}$ be the eigenvalues of $\triangle_i$ on $L^2(M)$, 
which by the representation theory of $\SL(2, \R)$, 
coincides with the positive eigenvalues of $\Box_i$ on $L^2(M)$.  
\begin{theorem}[Flaminio-Forni, \cite{FF1}]\label{theo:inv-dist}
The space $\mathcal I(M)$ has infinite countable dimension.  
There is a decomposition  
$$
\mathcal I(M) = \bigoplus_{\mu \in \sigma_{pp}} \mathcal I_\mu \oplus \bigoplus_{n \in \Z^+} \mathcal I_n \oplus \mathcal I_c\,,
$$
where 
\begin{itemize} 
\item  for $\mu = 0$, the space $\mathcal I_0$ is spanned by the 
Haar measure on $M$; 
\item  for $0 < \mu < 1/4$, there is a splitting 
$\mathcal I_\mu = \mathcal I_\mu^+ \oplus \mathcal I_\mu^-$, 
where $\mathcal I_\mu^{\pm} \subset W^{-s}(M)$ 
if and only if $s > \frac{1 \pm \sqrt{1 - 4 \mu}}{2}$, 
and each subspace has dimension equal to the multiplicity of $\mu \in \sigma_{pp}$; 
\item for $\mu \geq \frac{1}{4}$, the space 
$\mathcal I_\mu \subset W^{-s}(M)$ if and only if $s > 1/2$, 
and it has dimension equal to twice the multiplicity of $\mu \in \sigma_{pp}$;
\item for $n \in \mathbb Z_{\geq 2}$, 
the space $\mathcal I_n \subset W^{-s}(M)$ 
if and only if $s > n/2$ and it has dimension equal to twice the multiplicity 
of $\mu = \frac{1}{4}(-n^2 + 2n) \in \text{spec}(\Box_i)$,  
\item the space $\mathcal I_c \subset W^{-s}(M)$ 
if and only if $s > 1/2$.  
It is defined on the continuous spectrum of $\triangle_i$ on $W^{-s}(M)$,   
and it has infinite countable dimension.   
\end{itemize}
\end{theorem} 

For $s > 1/2$, Theorem~1.4 of \cite{FF1} shows $\mathcal I(M)$ 
has a countable basis $\mathcal B^s$ 
of unit normed (in $W^{-s}(M)$), generalized eigenvectors 
for the geodesic flow $\{e^{t X_i}\}_{t \in \R}$. 

For any $s > 1$, 
let 
$$
\mathcal B_{+}^{s} := \bigcup_{\mu \in \sigma_{pp}\backslash \{\frac{1}{4}\}} \mathcal B^{s} \cap \mathcal I_\mu^{s}\,, 
$$ 
be a basis of $U_i$-invariant distributions for $\left(\bigoplus_{\mu \in \sigma_{pp} \backslash \{\frac{1}{4}\}} \mathcal I_\mu\right)\,.$ 
Let 
$$
\mathcal B_{-}^s := \left(\bigcup_{\mu \geq \frac{1}{4}} \mathcal B^{s} \cap \mathcal I_\mu^{s}\right) \backslash \mathcal B_+^s\,.
$$
be a basis of invariant distributions for the rest of principal series.  
It will also be convenient to define 
$$
\mathcal B_{1/4}^s := \mathcal B^s \cap \mathcal I_{1/4}^s\,.
$$

For $D \in W^{-s}(M)$, let 
$$
\mathcal S_D := 
\left\{
\begin{array}{lll} 
\frac{1 \pm \re\sqrt{1 - 4 \mu}}{2} & \text{ if } D \in \mathcal I_\mu^\pm , \mu > 0\,; \\
n/2 & \text{ if } D \in \mathcal I_n, n \in \mathbb Z_{\geq 2}\,; \\
1/2 & \text{ if } D \in \mathcal I_c\,.
\end{array}
\right.
$$
We remark that by Lemma~\ref{lemm:KS_gap}, 
\[
\inf \left\{\frac{1 - \re\sqrt{1 - 4 \mu}}{2} > 0 : \mu \in \text{spec}(\Box_i) \cap \R^+\right\} > 0\,.
\]

For $T \geq 1$, let $\log^+ T := \max\{1, \log T\}$.  
\begin{theorem}\label{theo:equidistribution}
Let $\alpha > 3d/2 + 1$ and $d \in \N_{\geq 2}$.  
Let $\Gamma \subset \PSL(2, \R)^d$ 
be a cocompact, irreducible lattice.  

Then there is a constant $C_\alpha := C_\alpha(\Gamma) > 0$ such that for all $(x, T) \in M \times \R_{\geq 1}$, 
and there are real numbers $\{c_\mathcal D(x, T)\}_{D \in B_{+}^\alpha \cup B_{-}^\alpha}$ 
and distributions $\mathcal D_{x, T}^{\alpha, 2}, \mathcal R_{x, T}^\alpha \in W^{-\alpha}(M)$ 
such that the following estimate holds.  

For all $f \in W^\alpha(M)$,  
\begin{align*}
\frac{1}{T} \int_0^T f \circ h_t(x) dt - \int_M f d\vol &= \sum_{D \in \mathcal B_+^\alpha} c_D(x, T) D(f) T^{-\mathcal S_D}  \\ 
& + \sum_{D \in \mathcal B_{-}^\alpha} c_D(x, T) D(f) T^{-1/2} \log^+ T \\
&  + \frac{\mathcal D_{x, T}^{\alpha, 2}(f) \log^+ T + R_{x, T}^\alpha(f)}{T}\,,
\end{align*}
where for all $(x, T) \in M \times \R_{\geq 1},$  
\begin{equation}\label{eq:control-coef}
\sum_{D \in \mathcal B_+^\alpha \cup \mathcal B_-^\alpha} |c_D(x, T)|^2 + \|\mathcal D_{x, T}^{\alpha, 2}\|_{-\alpha}^2  + \|R^\alpha_{x, T}\|_{-\alpha}^2 \leq C_{\alpha}\,.
\end{equation}
Additionally, we have the following lower bound.  
For every $D \in \mathcal B^\alpha$, there is a constant $C_\alpha' := C_\alpha(D) > 0$ 
such that for all sufficiently large $T \gg 1$, 
\begin{equation}\label{eq:cd-lower}
\Vert c_D(\cdot, T)\| \geq \left\{ 
\begin{aligned} &C_\alpha' & \text{ if } D \notin \mathcal B^\alpha_{1/4}\,; \\  
&C_\alpha' \log^+ T & \text{ if } D \in \mathcal B^\alpha_{1/4}\,.
\end{aligned}
\right.
\end{equation}
\end{theorem}
\begin{remark}\label{rem:1}
For sufficiently large $T \gg 1$, the upper bound for the above coefficients is sharp up to multiplication by $\log (T)$.  
\end{remark}

\section{Equidistribution of coordinate horocycle flows}

Our estimates are given in terms of Sobolev norms involving a finite number of derivatives.  
In what follows, we let  
$$
\begin{aligned}
s > 3d/2 + 1, & \ \text{ and } \ i \in \{1, 2, \ldots, &d\} \,.
\end{aligned}
$$  
Let $\mathcal  W_s(M)$ be the maximal domain of the operator $(I + \triangle_0)$ 
on $L^2(M)$ with inner product 
$$
\langle f, g\rangle_{\mathcal  W_{s}(M)} := \langle (I + \triangle_0)^s f, g\rangle_{L^2(M)}\,.
$$

Let $\pi:\PSL(2, \R) \to \mathcal U(\mathcal  W_s(M))$ be a unitary 
representation of $\PSL(2, \R)$ defined by 
\begin{equation}\label{equa:rep}
\pi(g) f\left(\Gamma (a, h, b)\right) = f\left(\Gamma (a, h g, b)\right)\,, 
\end{equation}
for any $(a, h, b) \in \PSL(2, \R)^{i - 1} \times \PSL(2, \R) \times\PSL(2, \R)^{d - i}$.  

Now let $d\pi$ be the derived representation of $\pi$.  
The representation $d\pi$ is related to the derived representation of the 
right regular representation by the following simple lemma.
\begin{lemma}\label{lemm:regular_to_pi}
Let $Q \in \sl(2, \R)$, and let 
$Q_i := (0, \ldots, 0, Q, 0, \ldots, 0) \in \sl(2, \R)^d$, 
where $Q$ is in the $i_{\text{th}}$-position.  
Then 
$$
d\pi(Q) = Q_i \,.
$$
\end{lemma}
\begin{proof}
Let $x\in M$ and $f \in C^\infty(M)$.  For any $t\in \R$,
$$
f(x (I^{i - 1}, \exp(t Q), I^{d - i})) = f(x \exp(t Q_i))\,.
$$
We conclude by differentiating at $t = 0$.
\end{proof}
Hence, 
\begin{equation}\label{equa:regular_to_pi}
U_i  = d\pi(U) \text{ and } (I + \triangle_i) f =  d\pi\left(I - X^2 - 1/2(U^2 + V^2)\right).
\end{equation}

Then with respect to a positive Stieltjes measure, $dm(\mu)$, 
the unitary representation $\pi$ has the following direct integral 
decomposition 
$$
\mathcal W_s(M) = \int_{\oplus \mu \in \text{spec}(\Box_i)} \mathcal H_{\mu, s} dm(\mu)
$$
where there Casimir element $\Box_i$ acts as the constant 
$\mu$ on each unitary representation space $\mathcal H_{\mu, s}$. 
Each representation space $\mathcal H_{\mu, s}$ is a direct sum of an at most countable 
number of irreducible unitary representation spaces.  

By irreducibility and (\ref{equa:regular_to_pi}), the vector fields 
$U_i, X_i$ and $V_i$ are \textit{decomposable} into the irreducible representations of $\pi$    
in the sense that 
\begin{equation}\label{equa:W^s-decomp}
W^{s}(M) = \int_{\oplus \mu \in \text{spec}(\Box_i)} \mathcal H_\mu^s\,,
\end{equation} 
where $\mathcal H_\mu^s \subset \mathcal H_{\mu, s}$ 
inherits the inner product from $W^s(M)$.  

The following lemma is a consequence of 
Theorem~2 of \cite{KS}, as described in Section 1.3 of that paper.  
\begin{lemma}[Kelmer-Sarnak, \cite{KS}]\label{lemm:KS_gap}
We have 
\begin{equation}\label{eq:Casimir-i}
\inf \mathrm{spec}(\Box_i) \cap \R^+ > 0\,.
\end{equation}  
\end{lemma}
\begin{proof}
For any $j \in \N\setminus\{0\}$, and for a given infinite dimensional 
representation $\rho$ of $\PSL(2, \R)^j$, 
let $p(\rho)$ be the infimum of all $p$ such that there 
is a dense set of vectors $v$ such that $\langle \rho(g) v, v \rangle$ 
is in $L^p(\PSL(2, \R)^j)$.  
Now the regular representation of $\PSL(2, \R)^d$ on $L^2(M)$ 
has a countable, orthogonal decomposition 
into irreducible, unitary representations $\rho_m$ of 
$\PSL(2, \R)^d$.  
We may write $\rho_m = \rho_{m_1} \otimes \cdots \otimes \rho_{m_d}$, 
where the $\rho_{m_i}$ are irreducible, unitary representations of $\PSL(2, \R)$ 
in either the principal series, the complementary series or the 
discrete series.  

We present the following special case of Theorem 2 of \cite{KS}.  
\begin{theorem}[Kelmer-Sarnak]\label{theo:KS}
Let $\Gamma \subset \PSL(2, \R)^d$ be an irreducible, co-compact lattice, and let $\rho_m$ be as above. Then for any $\epsilon > 0$, $p(\rho_m) < 6+\epsilon$, except for a finite number of $m$'s.  
\end{theorem}
Principal series and discrete series representations are tempered, 
that is, if $\rho_{m_j}$ is such a representation, then for all $j$, $p(\rho_{m_j}) = 2$.  
Complementary series representations can be parameterized by 
$\nu_j \in (0, \frac{1}{2})\cup (\frac{1}{2}, 1)$,  
where $p(\rho_{m_j}) = \max\{1/\nu_j, 1/(1-\nu_j)\} \in \R^+$, 
which is unbounded if \eqref{eq:Casimir-i} is not true.  
We also have $p(\rho_m) = \max_j p(\rho_{m_j})$,  
so Theorem~\ref{theo:KS} shows that 
for an at most finite number of exceptions $m$,  
\[
\max\{1/\nu_i, 1/(1-\nu_i)\} \leq p(\rho_m) < 7 \,.  
\]  
This implies the lemma.  
\end{proof}

\subsection{Cohomological equation}
For any $s > 0$, define 
$$
Ann^s(M) := \left\{f \in W^s(M) : D(f) = 0 \text{ for all } D \in \mathcal I^{s}(M)\right\}.
$$

As a consequence of Theorem 1.2 of \cite{FF1}, we derive  

\begin{theorem}\label{theo:coeqn}
For any $0 \leq r < s -  1$, there is a constant 
$C_{r, s} := C_{r, s}(\Gamma) > 0$ such that the following holds.  
Then for any $f\in Ann^s(M)$, 
there exists a function $g \in W^r(M)$, 
unique up to additive constants, 
such that
\[
U_i g = f 
\]
and  
\[
\| g \|_{r} \leq C_{r, s} \| f \|_{s}\,.
\]
\end{theorem}
\begin{proof}
First say $r \in \Z_{\geq 0}$, and let $s$ be as in the theorem.  
By \eqref{equa:regular_to_pi}, by \eqref{eq:Casimir-i}, 
by Theorem 1.2 of \cite{FF1} and by Theorem 2 and Section 1.3 of \cite{KS}, 
we have that for any $f \in Ann^s(M)$ there is a zero average function $g \in W^r(M)$ 
and a constant $C_{r, s} := C_{r, s}(\Gamma) > 0$ such that 
\[
U_i g = f 
\]
and 
\begin{equation}\label{equa:FF_estimate}
\|(I + \triangle_i)^r g\|  
\leq C_{r, s} \|(I + \triangle_i)^s f \|\,.
\end{equation}

Now fix $f$ and $g$ as in the theorem.  
Because $U_i$ commutes with $\triangle_0$, for any $\alpha \geq 0$, we have 
\[
U_i \triangle_0^{\alpha} g = \triangle_0^{\alpha} f\,,
\]
Then for any $\epsilon > 0$ and any $\beta \geq 0$, (\ref{equa:FF_estimate}) gives 
a constant $C_{\beta, \epsilon} := C_{\beta, \epsilon}(\Gamma) > 0$ such that   
\begin{equation}\label{equa:vect_commute}
\|(I + \triangle_i)^\beta \triangle_0^{\alpha} g\|
\leq C_{\beta, \epsilon} \|(I + \triangle_i)^{\beta + 1 + \epsilon} \triangle_0^{\alpha} f\|\,.
\end{equation}

Then 
\begin{align}
\Vert g \Vert_r^2  & = \langle (I + \triangle_i + \triangle_0)^{2r} g, g\rangle \notag \\ 
& = \sum_{n = 0}^{2r} {2r \choose n} \langle (I + \triangle_i)^{n} \triangle_0^{2r-n} g, g\rangle \,. \label{eq:g-triangle_i}
\end{align}
It follows from Lemma~\ref{theo:KS} that $\triangle_0 > 0$ on the subspace of 
zero average functions $W_0^s(M)$ of $W^s(M)$.  
By the spectral theorem, $\triangle_0^{\alpha}$ is defined on $W_0^s(M)$ for any $\alpha > 0$, 
and moreover, because $\triangle_0$ and $(I+\triangle_i)$ commute on $W_0^s(M)$, 
we get that for any $\alpha, \beta \geq 0$, 
$\triangle_0^{\alpha}$ and $(I+\triangle_i)^{\beta}$ commute on $W_0^s(M)$ as well.  

Then by \eqref{equa:FF_estimate}\,,
\begin{align}
\langle (I + \triangle_i)^{n} \triangle_0^{2r-n} g, g\rangle & = \langle (I + \triangle_i)^{n} \triangle_0^{r-n/2} g, \triangle_0^{r-n/2} g\rangle \notag \\ 
& =\Vert (I + \triangle_i)^{n/2} \triangle_0^{r-n/2} g\Vert^2 \notag \\ 
& \leq C_r \|(I + \triangle_i)^{n/2 + 1 + \epsilon} \triangle_0^{r-n/2} f\|^2\,. \label{eq:g-f:1}
\end{align}
Because $(I + \triangle_i)^{\epsilon/2}$ commutes with $\triangle_0$ on $W^s(M)$, we have 
\begin{align}
\eqref{eq:g-f:1} &  \leq C_r  \langle (I + \triangle_i)^{n + 1} \triangle_0^{2r-n} (I + \triangle_i)^{\epsilon/2}  f, (I + \triangle_i)^{\epsilon/2} f\rangle\,. \label{eq:g-f}
\end{align}
By a result of Nelson (Lemma~6.3 of \cite{N}),  
for any $\alpha, \beta \in \Z_{\geq 0}$, 
there is a constant $C_{\alpha + \beta} > 0$ such that 
\[
\vert (I + \triangle_i)^{\alpha} \triangle_0^{\beta}\vert \leq C_{\alpha + \beta} (I + \triangle)^{\alpha + \beta}\,.
\]
Therefore
\begin{equation}
\eqref{eq:g-f} \leq C_r \langle (I + \triangle)^{2r+1} (I + \triangle_i)^{\epsilon/2}  f, (I + \triangle_i)^{\epsilon/2} f\rangle \,.  \label{eq:g-f-3}
\end{equation} 
Next notice that $(I + \triangle_i)^{\epsilon/2}$ commutes with $(I + \triangle)^{2r+1}$, 
and the spectral theorem gives 
\[
(I + \triangle_i)^{\epsilon/2} \leq (I + \triangle)^{\epsilon/2}\,.
\]
So  
\[
\eqref{eq:g-f-3} \leq C_r \Vert f \Vert_{r+1+\epsilon}^2\,.
\]

It follows that 
\[
\eqref{eq:g-triangle_i} \leq C_r \Vert f \Vert_{r+1+\epsilon}^2\,.
\]
By interpolation, the same estimate holds for any $r\geq 0$, 
which completes the proof of Theorem~\ref{theo:coeqn}   
\end{proof}

\subsection{Proof of Theorem~\ref{theo:equidistribution}}
For all $(x, T) \in M \times \R^+$, 
write $\gamma_{x, T}$ as 
$$
\gamma_{x, T}(f) := \frac{1}{T} \int_0^T f \circ h_t(x) dt\,.  
$$
We may orthogonally project $\gamma_{x, T}$ in $W^{-s}(M)$ 
onto a basis of $\{h_t\}_t$-invariant 
distributions described in Theorem~\ref{theo:inv-dist}.  
Let $\mathcal C_{\gamma_{x, T}}$ be orthgonal projection in $W^{-s}(M)$ 
of $\gamma_{x, T}$ into $\langle \bigcup_{D \in \mathcal B_-^s / \mathcal B_{1/4}^s}D\rangle.$  
For each $D \in \mathcal B^{s} \cap \langle \mathcal C_{\gamma_{x, T}}\rangle^\bot$, 
let $D_{\gamma_{x, T}}$ be orthogonal projection in $W^{-s}(M)$ of $\gamma_{x, T}$ 
onto $\langle D\rangle$.  
Then there is a remainder $R_{\gamma_{x, T}}$ such that 
\begin{equation}\label{equa:gamma-decomp}
\gamma_{x, T} = \left(\sum_{D \in \mathcal B^{s}\cap \langle \mathcal C_{\gamma_{, T}}\rangle^{\bot}} D_{\gamma_{x, T}}\right) \oplus \mathcal C_{\gamma_{x, T}} \oplus R_{\gamma_{x, T}}\,. 
\end{equation}
Notice that the space of invariant distributions in each irreducible, 
unitary representation is at most two dimensional.  
So Lemma 5.2 of \cite{FF1} gives that for some constant 
$C_s  > 0$, the quantity 
\begin{equation}\label{equa:norm_coefficients}
 \sum_{D \in \mathcal B^{s}(M)} \|D_{\gamma_{x, T}}\|_{-s}^2 
 + \|\mathcal C_{\gamma_{x, T}}\|_{-s}^2 + \|R_{\gamma_{x, T}}\|_{-s}^2 \,.
\end{equation}
satisfies 
\begin{equation}\label{equa:decom-like-gamma}
C_s^{-2} \|\gamma_{x, T}\|_{-s}^2 \leq (\ref{equa:norm_coefficients}) \leq C_s^2 \|\gamma_{x, T}\|_{-s}^2\,.
\end{equation}    

We prove Theorem \ref{theo:equidistribution} by estimating  
each of the terms in (\ref{equa:norm_coefficients}).  

\begin{lemma}\label{lemm:remainder}
Let $s > \frac{3d}{2} + 1$.  There is a constant 
$C_s := C_s(\Gamma) > 0$ such that for any $x \in M$ and any $T > 0 $, 
$$
\|R_{\gamma_{x, T}}\|_{-s} \leq \frac{C_s}{T}\,.
$$
\end{lemma}
\begin{proof}
Let $f \in Ann^s(M)$.  
Then by Theorem \ref{theo:coeqn}, 
for any $\frac{3 d}{2} < r < s - 1$, there is a constant $C_{r, s} := C_{r, s}(\Gamma) > 0$ 
and a function $g \in W^s(M)$ satisfying
$
U_i g = f
$
and 
$$
\| g \|_r \leq C_{r, s} \| f \|_s\,.
$$

Then as in Lemma 5.5 of \cite{FF1}, we get by the Sobolev embedding theorem 
that there is a constant $C_r := C_r(\Gamma) > 0$ such that 
\begin{align*}
|R_{{\gamma_{x, T}}}(f)| = & \frac{1}{T} |\int_0^T f \circ h_t(x) dt| \\
& = \frac{1}{T} |\int_0^T U_i g \circ h_t(x) dt| \\
& = \frac{|g \circ h_T(x) - g (x)|}{T} \\
& \leq \frac{C_r}{T} \| g \|_r \leq \frac{C_{r, s}}{T} \| f \|_s\,.
\end{align*}
\end{proof}

\begin{lemma}\label{lemm:continuous-component}
For every $s > \frac{3d}{2} + 1$, there is a constant $C_{s} := C_s(\Gamma) > 0$ such that the following holds.  
For any $\mu \in \sigma_{pp} / \{\frac{1}{4}\}$, for any $D \in \mathcal I_\mu^{\pm} \cap \mathcal B^s$, and for any $x \in M$ and $T > 1$, 
the distribution $D_{\gamma_{x, T}}$ from \eqref{equa:gamma-decomp} satisfies 
$$
\|D_{\gamma_{x, T}}\|_{-s} \leq C_{s} T^{-\mathcal S_D} \,.  
$$
\end{lemma}
\begin{proof}
Using Lemma~\ref{lemm:regular_to_pi}, 
the argument is the same as in Section~5.3 of \cite{FF1}.  
We give it here for the convenience of the reader.  
 
For any $x\in M$, for any $T \geq 1$ and for any $t \in \R$, we have 
\begin{equation}\label{equa:geodesicmap_gamma}
e^{t X_i} \gamma_{x, T} = \gamma_{x e^{-t X_i}, T e^{t}} \,.
\end{equation}
Then fix $x$ and $T$ as in the lemma, and note $e^{-\log T X_i} \gamma_{x, T} = \gamma_{x e^{\log T X_i}, 1}$\,.  
By Theorem~\ref{theo:inv-dist}, $\mathcal D_{x, T}$ is a generalized eigendistribution for the geodesic flow.  
Because the orthogonal splitting 
$\mathcal I^s(M) \oplus \mathcal I^s(M)^\bot$
is not preserved under $\{e^{t X_i}\}_t$,  
we do not immediately get an estimate of  $\Vert \mathcal D_{x, T}\Vert_{-s}$.  
Instead, we estimate it by an iterative argument. 
 
Let $h \in [1, 2]$ be such that $e^{h \lfloor \log^+T \rfloor} = T$.  
Using \eqref{equa:geodesicmap_gamma}, 
for any $l \in \{0, \cdots, \lfloor \log^+ T \rfloor-1\}$, we have
\begin{equation}\label{equa:iterative}
\begin{aligned}
\Vert D_{\gamma_{x e^{(\log T - (l+1) h)  X_i}, e^{(l+1)h}}} \Vert_{-s} 
& \leq \Vert \exp(hX_i) \mathcal D_{\gamma_{x e^{(\log T-l h) X_i}, e^{l h}}}\Vert_{-s} \\ 
&  + \Vert\left(\exp(h X_i) \mathcal R_{x e^{(\log T-l h) X_i}, e^{l h}} \right)\Vert_{-s} 
\end{aligned}
\end{equation}
Now by Theorem~1.4 of \cite{FF1} and Lemma~\ref{lemm:regular_to_pi},   
$D$ is an eigendistribution of $e^{h X_i}$ with eigenvalue $e^{-h(1 \pm\sqrt{1 - 4 \mu})/2} = e^{-h \mathcal S_D}$. 
Hence, the same is true for $\mathcal D_{\gamma_{x e^{(\log T-lh) X_i}, e^{lh}}}$, for any $l$.  
Also, $e^{h X_i}$ is a bounded operator.  
So there is a constant $C > 0$ such that 
\begin{equation}\label{eq:iterate_2}
\eqref{equa:iterative}  \leq e^{-h\mathcal S_D} \Vert \mathcal D_{\gamma_{x e^{(\log T-l h) X_i}, e^{l h}}}\Vert_{-s} + C \Vert\mathcal R_{x e^{(\log T-l h) X_i}, e^{l h}} \Vert_{-s}\,.
\end{equation}
Recall that $(\gamma_{x , T})_{|_{\langle D \rangle}} = D_{x , T}$, 
so we iterate and get 
\begin{equation}\label{eq:iterate-T}
\begin{aligned}
\Vert \left(\gamma_{x , T}\right)_{|_{\langle D \rangle}}\Vert_{-s} &\leq T^{-\mathcal S_D} \| \mathcal D_{\gamma_{x e^{\log T X_i}, 1}} \|_{-s}^2  \\ 
& + C T^{-\mathcal S_D} \sum_{l = 1}^{\lfloor \log^+T\rfloor} e^{lh\mathcal S_D}  \Vert\mathcal R_{x e^{(\log^+ T-l h) X_i}, e^{l h}} \Vert_{-s}\,.
\end{aligned}
\end{equation}
Moreover, Lemma~\ref{lemm:remainder} gives a constant $C_s := C_s(\Gamma) > 0$ such that for any $l$, 
\[
\Vert \mathcal R_{x e^{(\log^+T-lh) X_i}, e^{lh}}\Vert_{-s} \leq C_s e^{-l h}\,.
\] 
Because $\mu \in \sigma_{pp} / \{\frac{1}{4}\}$, 
the series in \eqref{eq:iterate-T} is bounded by a constant depending only on $s$ and $\Gamma$.     
\end{proof}

\begin{proof}[Proof of Theorem~\ref{theo:equidistribution}] 
By Lemma~\ref{lemm:continuous-component}, it remains to 
prove the upper bounds for distributions in $\mathcal B^s\cap(\mathcal I_c \cup \mathcal I_{1/4})$ and $\mathcal B^s \cap \mathcal I_n,$ for $n \in \N_{\geq 2}$.  
By Lemma~5.1 of \cite{FF1} and Lemma~\ref{lemm:regular_to_pi}, 
there is a constant $C_s > 0$ such that for all $t \in \R$, 
\begin{equation}\label{eq:log-loss}
\Vert \exp(t X_i) D \Vert_{-s} \leq C_s(1 + \vert t \vert) e^{-t/2}\,.
\end{equation}
Then by replacing $e^{-(1 \pm\re\sqrt{1 - 4 \mu})/2} $ with $(1+\vert h \vert)e^{-1/2} $ in formula \eqref{eq:iterate_2} 
of the above argument, we deduce that there is a constant $C_{s} > 0$ such that 
\[
\|D_{\gamma_{x, T}}\|_{-s} \leq C_{s} T^{-(1\pm\re\sqrt{1 - 4 \mu})/2} \log^+T \,.  
\]

For $D \in \mathcal B^s \cap\mathcal I_{1/4}$, Theorem~1.4 of \cite{FF1} and Lemma~\ref{lemm:regular_to_pi} show that $D$ is a generalized eigendistribution for $e^{h X_i}$ satisfying \eqref{eq:log-loss} for all $t \in \R$.  

For $D \in \mathcal B^s\cap \mathcal I_{n}$ and $n \in \N_{\geq 2}$, 
and for $h$ as in Lemma~\ref{lemm:continuous-component}, 
Theorem~1.4 of \cite{FF1} gives that $D$ is a eigendistribution 
for $\exp(h X_i)$ with eigenvalue $e^{-n h/2}$.  
Then the above argument gives,
 for any $x \in M$ and any $T \geq 1$,  
\begin{align}
&\|D_{\gamma_{x, T}}\|_{-s} \leq C_{s} T^{-1} \log^+T& \text{ if }n = 2 \,; \label{equa:discrete_inv-est-1} \\ 
&\|D_{\gamma_{x, T}}\|_{-s} \leq C_{s} T^{-1} & \text{ if }n \in \N_{\geq 3} \,.\label{equa:discrete_inv-est-2}
\end{align}
Now we define the remainder distribution $\mathcal R_{x, T}^s$ appearing in Theorem~\ref{theo:equidistribution} as the orthogonal 
sum of the distribution $R_{\gamma_{x, T}}$ and the distributions $D_{\gamma_{x, T}}$ from \eqref{equa:discrete_inv-est-2}.    
The estimate of $\Vert \mathcal R_{x, T}^s\Vert_{-s}$ follows from Lemma~\ref{lemm:remainder}, formula \eqref{equa:discrete_inv-est-2} and orthogonality.     
Lastly, for $n = 2$, we define 
\[
 \mathcal D_{x, T}^{s, 2} := D_{\gamma_{x, T}}\,,
\]
so \eqref{equa:discrete_inv-est-1} gives the estimate of $\Vert \mathcal D_{x, T}^{s, 2} \Vert_{-s}\,.$  
This concludes the proof for the upper bounds of the distributions 
in Theorem~\ref{theo:equidistribution}.  

The $L^2$ lower bounds can be obtained by an argument  
involving the $L^2$ version of the Gottschalk-Hedlund Lemma.  
Using Lemma~\ref{lemm:regular_to_pi}, the lower bound follows from  Lemma~5.7, Lemma~5.8, Lemma~5.9 and 
Lemma~5.13 of \cite{FF1}.  
We give that argument here for 
the convenience of the reader.  
  
The Gottschalk-Hedlund Lemma says, 
in particular, that if $f$ is not a coboundary for $\{h_t\}_{t}$, 
then the family of functions $\{T\gamma_{x, T}(f)\}_{T \geq 1}$ on $M$ is not 
equibounded in the $L^2$-norm.  

So let $D \in \mathcal B^s_{\pm}$.  Then we can find a function $f \in W^s(M) \cap Ann^s(M)^{\bot}$ such that 
$D(f) = 1$ and for all $\bar D \in \mathcal B_{\pm}^s \setminus \langle D \rangle$, $\bar D(f) = 0$.  
So for all $x \in M$ and $T \geq 1$, 
\[
\gamma_{x, T}(f) = c_{D}(x, T)\,.
\]
Hence, 
\begin{equation}\label{eq:cd-T}
\sup_{T \geq 1} T \Vert c_{D}(\cdot, T) \Vert = +\infty\,.
\end{equation}

If $D \notin \mathcal B_{1/4}^s$, then for any $m \in \N$, 
\[
\Vert c_D(\cdot, T e^{m}) \Vert \geq e^{-m \mathcal S_D} \Vert c_D(\cdot, T) \Vert  - \mathcal E_D^s(x, T,m) \,,
\]
where $\mathcal E_D^s(x, T, m)$ is the contribution of the remainder given by 
\[
\begin{aligned}
\mathcal E_D^s(x, T, m) := & C_s e^{-m\mathcal S_D}  \sum_{l = 1}^{m} e^{l\mathcal S_D}  \Vert\mathcal R_{x e^{(\log^+ T +(m-l)) X_i}, T e^{l}} \Vert_{-s} \\ 
& \leq \frac{C_s}{T} e^{-m\mathcal S_D} \,,
\end{aligned}
\]
for some constant $C_s > 0$.  

By \eqref{eq:cd-T}, there is some $T > 1$ such that 
\[
\Vert c_D(\cdot, T) \Vert \geq 2 \frac{C_s}{T}\,.
\]
This implies \eqref{eq:cd-lower} in the case 
$D \notin \mathcal B_{1/4}^s$.  
The argument for $D \in \mathcal B_{1/4}^s$ is similar, 
see formulas (123) and (124) of \cite{FF1} for details. 
\end{proof}

\begin{proof}[Proof of Remark~\ref{rem:1}]
The proof is essentially given by Corollary 5.17 of \cite{FF1}.  
Let $\mu_0$ be the bottom of the positive spectrum of $\Box_i$ 
on $L^2(M)$.  
From the decomposition in Theorem~\ref{theo:equidistribution} 
it is enough to consider the projection of $\gamma_{x, T}$ onto the 
invariant distributions supported near $\mu_0$ in the spectrum of $\Box_i$.  

First suppose that $\mu_0 \in \sigma_{pp}$. 
Then it is enough to consider $D \in \mathcal I_{\mu_0}$,  
and let $c_D$ be the coefficient for $D$ in the decomposition 
from Theorem~\ref{theo:equidistribution}.  
Because pointwise upper bounds for $c_D$ are the same order in $T$ 
as the $L^2$-lower bounds, 
we have a constant $C_s := C_s(D) > 0$ 
such that for any $x \in M$ 
and any $T \gg 1$ sufficiently large, 
\begin{equation}\label{equa:C_D-bound-|-0}
\vert c_D(x, T)\vert \leq C_s \Vert c_D(\cdot, T)\Vert\,.
\end{equation}
Now we sketch the argument in Corollary~5.17 of \cite{FF1}, 
which proves the estimate in Theorem~\ref{theo:equidistribution} is sharp.  

Let $K \in (0, 1)$ be a constant, and let $A_T := A_{D, T, K}  \subset M$
be defined by 
$$
A_{T} := \left\{x \in M : \vert c_{D}(x, T)\vert > K \Vert c_D(\cdot, T)\Vert\right\}\,.
$$
Using \eqref{equa:C_D-bound-|-0}, 
it follows that 
$$
\left(K^2 + (C_s)^2 \vol(A_T)\right) \Vert c_D(\cdot, T)\Vert_0^2 \geq \Vert c_D(\cdot, T)\Vert^2\,.
$$
It follows that there is a constant $\alpha := \alpha_{s, K, D} > 0$ such that for all $T \gg 1$ 
sufficiently large, 
$$
\vol(A_T)\geq \alpha\,.
$$
Now the $L^2$ lower bounds on $c_D$ prove the remark in this case.  

If $\mu_0$ is contained in the continuous spectrum of $\Box_i$,  
then $\mu_0 \geq \frac{1}{4}$, 
and by Theorem~\ref{theo:inv-dist}, the space $\mathcal I_c$ 
is infinite dimensional.   
So there is a distribution $D \in \mathcal B_-^s \cap \mathcal I_c$ 
that is supported away from $\mu = \frac{1}{4}$.  
Because $D$ is a direct integral of distributions with Casimir parameter $\mu >  \frac{1}{4}$, Theorem~1.4 of \cite{FF1} and Lemma~\ref{lemm:regular_to_pi} give 
$$
\Vert \exp(t X_i) \exp(\mathcal D_{\gamma_{x e^{\log T X_i}, 1}}) \Vert_{-s} = e^{-t/2} \Vert \mathcal D_{\gamma_{x e^{\log T X_i}, 1}} \Vert_{-s} \,.
$$
Then, as in Lemma~\ref{lemm:continuous-component}, we get  
that $D_{\gamma_{x, T}}$ is sharper than the estimate in Theorem~\ref{theo:equidistribution} by a logarithmic factor.  
There is a constant $C_{s, \Gamma} > 0$ such that for any $x \in M$ 
and any $T \geq 1$, 
$$
\|D_{\gamma_{x, T}}\|_{-s} \leq C_{s, \Gamma} T^{-1/2} \,,
$$
Hence, the pointwise upper bound on the coefficient $c_D$ from Theorem~\ref{theo:equidistribution} is the same order in 
$T$ as its $L^2$ lower bound.  
The above argument implies that the pointwise upper 
bound for $\vert c_D(x, T)\vert $ is sharp 
on a set $A_T$ of positive measure.  

Comparing this bound with the upper bounds from all coefficients $c_D$ in Theorem~\ref{theo:equidistribution} proves Remark~\ref{rem:1}.  
\end{proof}

\section{Acknowledgements}
The author is grateful to Giovanni Forni for helpful discussions and encouragement.  I would also like to thank the referee for helpful comments.


\begin{thebibliography}{11}

\bibitem{BEG} M. Bj$\ddot{\text{o}}$rklund, M. Einsiedler, A. Gorodnik. \emph{Quantitative Multiple Mixing}.  arXiv:1701.00945

\bibitem{BV} T. Browning, Ilya Vinogradov. \emph{Effective ratner theorem for $\SL(2, \R) \ltimes \R^2$ and gaps in $\sqrt{n}$ modulo 1}.  http://arxiv.org/pdf/1311.6387v2.pdf 

\bibitem{EMV} M. Einsiedler, G. Margulis, A. Venkatesh. \emph{Effective equidistribution for closed orbits of semisimple groups on homogeneous spaces}.  Invent. Math. \textbf{177} (2009), 137-212.  

\bibitem{FF1} {L. Flaminio, G. Forni}.  \emph{Invariant Distributions and Time Averages for Horocycle Flows}. Duke J. of Math, {\bf 119},  (2003), No 3, 465--526.

\bibitem{FF3} L. Flaminio, G. Forni.  \emph{On effective equidistribution for higher step nilflows}. arXiv:1407.3640 

\bibitem{FFT} L.Flaminio, G. Forni, J. Tanis. \emph{Effective equidistribution of twisted horocycle flows and horocycle maps}. arXiv:1412.5353  

\bibitem{GT} B. Green, T. Tao. \emph{The quantitative behavior of polynomial orbits on nil manifolds}, Annals of Math, \textbf{175} (2012), 465-540.  

\bibitem{KS} D. Kelmer, P. Sarnak. \emph{Strong spectral gaps for compact quotients of products of }$\mathit{PSL(2, \R)}$.  J. Eur. Math. Soc. 11, (2009) 283 - 313.

\bibitem{L} J. Lions.  \emph{Non-homogeneous boundary value problems and applications. Vol 1.} Springer-Verlag New York, New York, 1972.

\bibitem{N} E. Nelson. \emph{Analytic Vectors}, Annals of Mathematics, Second Series, Vol. 70, No. 3, 572-615.

\bibitem{R1} M. Ratner, \emph{On Raghunathan's measure conjecture}, Annals of Math. \textbf{134} (1991) 545-607.

\bibitem{R2} M. Ratner, \emph{Raghunathan's topological conjecture and distributions of unipotent flows}, Duke Math. J. \textbf{63} (1991), 235-280.   

\bibitem{St} A. Strombersson. \emph{An Effective Ratner Equidistribution Result for $\SL(2,\R)\ltimes \R^2$}.  Duke Math. J. 164 (2015), 843-902.  

\bibitem{TV} J. Tanis, P. Vishe.  \emph{Uniform Bounds for Period Integrals and Sparse Equidistribution}. International Mathematics Research Notices 2015;
doi: 10.1093/imrn/rnv115.

\bibitem{V} A. Venkatesh.  \emph{Sparse equidistribution problems, period bounds and subconvexity}.  Ann. of Math. 172 : 989-1094, 2010. 

\bibitem{IV} I. Ilya Vinogradov. \emph{Effective equidistribution of horocycle lifts} arXiv:1607.04769


\end{thebibliography}
\end{document}